\documentclass[a4paper, 11pt, oneside]{memoir}
\usepackage{memhfixc} 

\setulmarginsandblock{\uppermargin}{\lowermargin}{*}
\setheaderspaces{*}{\headsep}{*} 
\checkandfixthelayout[nearest] 

\nonzeroparskip 
\abstractrunin 
\abslabeldelim{.\quad} 
\predate{} 
\date{}
\postdate{}

\setsecnumdepth{subsection} 
\counterwithout{section}{chapter} 
\counterwithout{figure}{chapter} 
\namerefoff 
\setparaheadstyle{\bfseries\addperiod} 

\setfloatlocations{figure}{!ht} 
\captionnamefont{\bfseries}
\captiontitlefont{\itshape}
\hangcaption

\usepackage[utf8]{inputenc}
\usepackage[T1]{fontenc}
\usepackage{amssymb}
\usepackage[main=UKenglish, ngerman, polish]{babel}
\usepackage{array}
\usepackage[utf8]{inputenc}
\usepackage{amsmath}
\usepackage{amsfonts}
\usepackage{amstext}
\usepackage{amsthm}
\usepackage{enumerate}
\usepackage{relsize}
\usepackage[T1]{fontenc}
\usepackage{a4wide}
\usepackage{makeidx}
\usepackage{textcomp}
\usepackage{geometry}
\usepackage{sverb}
\usepackage{caption}
\usepackage{listings}
\usepackage[T1]{fontenc}
\usepackage{calligra}
\usepackage{csquotes}
\usepackage{dsfont}

\usepackage{graphicx}
\usepackage{xcolor}
\usepackage{tikz}
\usepgflibrary{decorations.pathreplacing}
\usepackage{pgf}
\usepackage{pgfkeys}
\usepackage{pgffor}
\usepackage{pgfpages}
\numberwithin{equation}{section}    
\usepackage{color}
\definecolor{hellgrau}{rgb}{0.93,0.93,0.93}
\definecolor{hellergrau}{rgb}{0.97,0.97,0.97}
\definecolor{grau}{rgb}{0.93,0.93,0.93}
\definecolor{hellblau}{rgb}{0.8,0.8,2.0}
\definecolor{blau}{rgb}{0.3,0.5,2.0}
\definecolor{hellrot}{rgb}{2.0,0.6,0.6}
\definecolor{gruen}{rgb}{0.3,0.75,0.2}
\definecolor{rot}{rgb}{0.9,0.1,0.1}

\addto\captionsngerman{\renewcommand{\refname}{Literaturverzeichnis}}
\newcommand{\IR}{{\mathbb{R}}}
\newcommand{\IN}{{\mathbb{N}}}

\newcommand{\IQ}{{\mathbb{Q}}}
\newcommand{\IC}{{\mathbb{C}}}

\newcommand{\I}{{\mathds{1}}}
\newcommand*{\close}[1]{\overline{#1}} 
\newcommand*{\charfkt}[1]{\mathds{1}_{#1}} 
\newcommand*{\defas}{:=}
\newcommand*{\asdef}{=:}

\newcommand*{\G}{G} 
\newcommand*{\g}{t} 
\newcommand*{\subG}{K} 
\newcommand*{\identif}{\mu} 
\newcommand*{\Cck}{C_c(\subG)} 
\newcommand*{\Cc}{C_c(\G)} 
\newcommand*{\CC}{C_c(\IC)} 
\newcommand*{\C}[1]{C(#1)} 
\newcommand*{\cc}[1]{C_c(#1)} 
\newcommand{\Cu}{C_u(\IR)}

\newcommand*{\nset}{M} 
\newcommand*{\act}{\alpha} 
\newcommand*{\trans}{\tau}
\newcommand*{\MG}{\mathcal{M}(\G)} 
\newcommand*{\MGp}{\mathcal{M}^{+}(\G)} 
\newcommand*{\MGtb}{\mathcal{M}^{\infty}(\G)} 

\newcommand*{\gelf}{J} 

\newcommand*{\oneset}{\Omega} 
\newcommand*{\twoset}{X} 
\newcommand*{\threeset}{\Omega_{\mu}} 

\newcommand*{\onelem}{\omega} 
\newcommand*{\twoelem}{x} 
\newcommand*{\threelem}{\mu_{x}} 
\newcommand*{\threelemvari}[1]{\mu_{#1}} 

\newcommand*{\mone}{m} 
\newcommand*{\mtwo}{m'} 
\newcommand*{\mthree}{m_{\identif}} 

\newcommand*{\LtwoG}{L^{2}(X,m)} 
\newcommand*{\Loneinf}{L^{\infty}(\oneset,\mone)} 
\newcommand*{\Lonetwo}{L^{2}(\oneset,\mone)} 
\newcommand*{\Ltwo}{L^{2}(\twoset,\mtwo)} 
\newcommand*{\Ltwoinf}{L^{\infty}(\twoset,\mtwo)}
\newcommand*{\Lthree}{L^{2}(\threeset,\mthree)} 
\newcommand*{\Lthreeinf}{L^{\infty}(\threeset,\mthree)} 

\newcommand*{\fktG}{\varphi} 
\newcommand*{\fktC}{\vartheta}
\newcommand*{\fktGseq}[1]{\fktG_{#1}} 
\newcommand*{\fktCseq}[1]{\widetilde{\fktC}_{#1}} 
\newcommand*{\fktCSeq}[1]{\widehat{\fktC}_{#1}} 
\newcommand*{\fix}{L} 

\newcommand*{\Nrx}{\widetilde{N}'} 
\newcommand*{\Nx}{N'} 
\newcommand*{\unit}{U_{\identif}} 
\newcommand*{\Uthree}{U_{\identif}} 
\newcommand*{\Identif}{\identif_{\#}} 
\newcommand*{\Nthree}{N_{\identif}} 
\newcommand*{\ftwo}{f} 
\newcommand*{\fthree}{f} 
\newcommand*{\Unit}{U} 
\newcommand*{\Three}{T^{\identif}} 
\newcommand*{\Two}{T'} 
\newcommand*{\actwo}{\alpha'} 
\newcommand*{\acthree}{\alpha^{\identif}} 
\newcommand*{\Nrist}{N|_{\Cck}} 

\newcommand*{\impl}{\Rightarrow} 

\newcommand*{\iso}{\cong}
\newcommand{\too}{\longrightarrow}
\newcommand\absinf[1]{\ensuremath{\|#1\|_{\infty}}}
\newcommand\abstwo[1]{\ensuremath{\|#1\|_{L^2}}}
\newcommand{\lin}{\operatorname{Lin}}
\newcommand{\sub}{\subset}

\newcommand{\A}{\mathcal{A}} 
\newcommand{\N}{\mathcal{N}} 
\newcommand{\Ahat}{\hat{\A}} 
\renewcommand*{\a}{a} 
\newcommand*{\ahat}{\hat{\a}} 
\newcommand{\B}{\mathcal{B}}

\newcommand\abs[1]{\ensuremath{\vert#1\vert}}

\newcommand\norm[1]{\ensuremath{\|#1\|}}

\newcommand{\supp}{\operatorname{supp}}
\renewcommand{\phi}{\varphi}
\renewcommand{\epsilon}{\varepsilon}
\renewcommand{\upsilon}{\vartheta}

\newtheorem{theorem}{Theorem}[section]
\newtheorem{lemma}[theorem]{Lemma}
\newtheorem{prop}[theorem]{Proposition}
\newtheorem{remark}[theorem]{Remark}
\newtheorem{definition}{Definition}[section]

\vspace{20pt}
        \large{

\title{Characterization of Translation Bounded Measure Dynamical Systems and Positive Measure Dynamical Systems within the Spatial Processes }
  }            
               
\author{Franziska Sieron}
\begin{document}

\thispagestyle{empty}
\maketitle
\begin{abstract}
In this paper we consider spatial processes and measure dynamical systems over locally compact Abelian groups. We characterize when a spatial processes is equivalent to a translation bounded measure dynamical systems and we characterize when a spatial processes is equivalent to a positive measure dynamical systems. The basic idea of our approach is the identification of the elements of a spatial process with translation bounded measures respectively positive measures by applying the Gelfand theory of $C^*$-algebras.
\end{abstract}

\section{Introduction}

Since the discovery of quasicrystal by Schechtman in 1982 \cite{Sche} the field of aperiodic order has received a lot of attention, see e.g. the monographs and survey collections \cite{BaGr_ApOrd_I,BaGr_ApOrd_II,BaMo00,KeLeSa,Moo,Pat}.

In this context, the modeling via dynamical systems has a long history.
In one dimension this leads to symbolic dynamics. A very famous class of examples comes from primitive substitutions. Cut-and-project schemes provide another large class of examples. These schemes can be used to construct examples in higher dimensions as well see e.g. \cite{SchCPS,SchGM}. This leads to tiling dynamical systems and Delone dynamical systems see e.g. \cite{Sol}. Specifically, Delone sets provide a mathematical abstraction of the position of the atoms in quasicrystals see e.g. \cite{Lag99}. Another method to model quasicrystals is given in \cite{Bak} where the aperiodic distribution of quasicrystals is described via (continuous) almost periodic density functions.
Moreover certain point processes have also been considered in \cite{BaBiMo,GouQC,GouDif}.
These descriptions are quite different. However, they can all be put under a common umbrella of measure dynamical systems.
In fact two classes of measures have played a role:\\
One is the class of translation bounded measures. These measures are not necessarily positive but they have a uniform boundedness property. The measures coming from Dirac combs of Delone sets, which are frequently considered in diffraction theory of dynamical system, are included in this class, see e.g. \cite{BaMo04,LeMoSo}. This leads to translation bounded measure dynamical systems (TBMDS). The approach to these systems has been presented in \cite{BaLe04}.
The other class are positive measures. This class includes the stochastic point processes of \cite{BaBiMo,GouQC,GouDif}.
Here we can drop the boundedness condition in favor to the positivity of the measures. The according systems, the positive measure dynamical systems (PMDS), are discussed in \cite{LeSt}.

In 2009 Lenz and Moody introduced in \cite{LeMo} the concept of spatial processes and developed a diffraction theory for them. The spatial processes of \cite{LeMo} include both the TBMDS and the PMDS.
In fact the concept introduced by them is even more general and does not need dynamical systems but just the representation of groups. \\
This naturally leads to the question of a characterization of TBMDS and PMDS within the spatial processes.
The aim of this article is to give an answer to this question, which will be provided in the two main theorems of the article. The first theorem yields a characterization of TBMDS, while the second theorem yields a characterization of PMDS within the spatial processes.

The article is organized as follows:
In the first section we introduce the different concepts properly and give precise definitions. The second section deals with Theorem \ref{ThmEquiTBMDS}, which states a characterisation of TBMDS. The third section considers Theorem \ref{equiPMDS}, which states a characterisation of PMDS. 
The last section, the appendix, shows that some apparently different definitions in the literature in fact agree. We also show
a more technical statement about the convergence of the convolution of functions.


\section{Preliminaries}
In this section, we introduce notations and concepts used throughout this article. In particular, we introduce the concept of stationary processes.

Let $\G$ be a locally compact, $\sigma$-compact, abelian group and $X$ a compact topological Hausdorff space space. Let 
$$\act:\G\times X\too X,~(\g,x)\mapsto\act_{\g}(x)$$
be a continuous action of $\G$ on $X$, where $\G\times X$ carries the product topology. Then $(X,\act)$ is a \emph{topological dynamical system}. If $(X,\Sigma_{X})$ is a measurable space, where $\Sigma_{X}$ is a $\sigma$- algebra on $X$, and the action $\act$ is measurable in each variable, then we call $(X,\act)$ a \emph{ measurable dynamical system}.\\
We denote the vector space of complex valued continuous functions on $\G$ with compact support by $\Cc$.
For a compact subset $\subG\subset\G$, the set of complex valued continuous functions on $\G$ with support in $\subG$ is given by $\Cck$.
The space $\Cc$ is equipped with the locally convex limit topology induced by the canonical embedding $\Cck\hookrightarrow\Cc$. The translation by $\g\in\G$ on $\Cc$ is given by 
$$\trans_{\g}\fktG :=\fktG(\cdot-\g)$$
with $\fktG\in\Cc$. Let $\mone$ be a $\G$-invariant probability measure on $\twoset$. The corresponding set of equivalence classes of square integrable functions on $\twoset$ is denoted by $\LtwoG$. The action $\act$ induces a unitary representation $T$ of $\G$ on the Hilbert space $\LtwoG$, where for each $t\in\G$ the unitary operator
$$T_{\g}:\LtwoG\too\LtwoG ~\text{ is given by }~ T_{\g}f:= f\circ \act_{-t}.$$ 
This representation is also known as the \emph{Koopman representation} and $T_{\g}$ as \emph{Koopman operator} associated to $(X,\act)$  with measur $m$.

We now discuss spatial processes following \cite{LeMo}.
A \emph{spatial processes} $\N=(N,\twoset,\mone,T)$ consists of
\begin{itemize}
\item a dynamical system $(\twoset,\act)$ over the locally compact, abelian group group $\G$,
\item a $\G$-invariant probability measure $m$ on $\twoset$,
\item a linear and continuous map
$$N:\Cc\too\LtwoG \mbox{ with }\,\close{N(\fktG)}=N(\close{\fktG}),\mbox{ for all }\fktG\in\Cc,$$
satisfying $T_{\g}\circ N=N\circ\trans_{\g}$ for all $\g\in\G$.
\end{itemize}
A spatial process is called \emph{full}, if the linear span
$$\lin\{\fktC_1(N(\fktG_1))\cdots\fktC_n(N(\fktG_n)):~\fktC_j\in\ \CC,~\fktG_j\in\Cc~\mbox{for}~j=1,\cdots,n\text{ and }n\in\IN\}$$
is a dense subset of $\LtwoG$.

\begin{definition}\label{EquiProc}
Two spatial processes $\N=(N,X,m,T)$ and $\N’=(N',X’,m’,T’)$ are equivalent (or spatially isomorphic) if there exists a unitary map $\Unit:L^2(X,m)\too L^2(X’,m’)$ with
\begin{itemize}
\item $\Unit\circ N=N’,$
\item $\Unit\circ T_{\g}=T’_{\g}\circ\Unit$ for all $\g\in G,$
\item $\Unit(fg)=\Unit(f)\Unit(g)$ for all $f,g\in L^{\infty}(X,m)$ and $\Unit^{-1}(f’g’)=\Unit^{-1}(f’)\Unit^{-1}(g’)$ for all $f’,g’\in L^{\infty}(X’,m’).$
\end{itemize}
The map $\Unit$ is called spatial isomorphism between the processes.
\end{definition}

\begin{remark}
Note that the map $U^{-1}$ has analogous properties with the roles of $N$ and $N'$ reversed. Hence, spatially isomorphy is indeed an equivalence relation. 
\end{remark}

The following statement is probably known but does not seem to appear in the literature so far. 
First we recall that a measurable dynamical system is called \emph{ergodic} if every measurable function $f$ satisfying $f=T_{\g}f$ for all $\g\in\G$ is constant almost surly.

\begin{theorem}
Let $\N=(N,\oneset,\mone,T)$ and $\N’=(\Nx,\twoset,\mtwo,\Two)$ be two equivalent spatial processes. Then $(\Nx,\twoset,\mtwo,\Two)$ is  ergodic if and only if $(N,\oneset,\mone,T)$ is ergodic.
\end{theorem}

\begin{proof}
We show the implication $\N$ egodic $\impl$ $\N'$ ergodic. The other implication follows clearly by interchanging $\N$ and $\N'$.

For this purpose we consider a
$g’\in\Ltwo$ with $\Two_{\g}g’=g’$ for all $\g\in\G$ and deduce that $g’$ is constant almost surly.

Let $\Unit$ be the spatial isomorphism between $\N$ and $\N’$. Let $\Unit^{-1}$ be its inverse and $f=\Unit^{-1}(g’)$. Then we have
$$T_{\g}f=T_{\g}\Unit^{-1}(g’)=\Unit^{-1}(\Two_{\g}g’)=\Unit^{-1}(g’)=f.$$
Consequently $f$ is constant almost surly since $\N$ is ergodic.
Since $\mtwo$ is a probability measure we have $\mathds{1}_{\twoset}\in\Ltwo$ and we obtain for a function $h’\in\Ltwoinf$ 
$$\Unit^{-1}(\mathds{1}_{\twoset})\Unit^{-1}(h’)=\Unit^{-1}(\mathds{1}_{\twoset}h’)=\Unit^{-1}(h’).$$
It follows that $\mathds{1}_{\oneset}=\Unit^{-1}(\mathds{1}_{\twoset})$ and thus $\Unit(\mathds{1}_{\oneset})=\mathds{1}_{\twoset}$.
Hence, we achieve for a constant $C$
$$\Unit(f)
=\Unit(f\mathds{1}_{\oneset})
=\Unit(C\mathds{1}_{\oneset})
=C~\Unit(\mathds{1}_{\oneset})
=C~\mathds{1}_{\twoset}
=C=g’ \mbox{~a.s.},$$
which finishes the proof.
\end{proof}

Next we introduce the two classes of examples of spatial processes,  which will be considered more closely in this article. These examples are in fact dynamical systems on $\MG$. The set $\MG$ is defined to be the dual of $\Cc$, i.e. the space of continuous linear functionals on $\Cc$. Such a functional can be considered as complex Borel measures by the Riesz-Markov theorem. Thus $\MG$ is a \emph{ set of measures on $G$} and carries the vague topology which equals the weak-$*$-topology. We give a brief introduction to these dynamical systems.

The first class is the class of dynamical systems on translation bounded measures (see \cite{BaLe04} for full details and proofs). 
Here we call a Borel measure $\nu$ on $\G$ \emph{translation bounded} if the map  
$$\G\to\IC,~t\mapsto {\int_{\G}\trans_{\g}\fktG \,\mathrm d\nu(s)}$$ 
is bounded for all $\fktG\in\Cc$ and denote the set of translation bounded measures on $\G$ by $\MGtb$. 
The action $\act$ on $\MG$ is given via
$\act:\G\times\MG\too\MG$ with 
$$\act_{\g}(\nu)(\fktG):=\delta_{\g}*\nu(\fktG)=\int_{\G}\fktG(t+s)\,\mathrm d\nu(s)=\int_{\G}\trans_{-\g}\fktG\,\mathrm d\nu(s)=\nu(\trans_{-\g}\fktG).$$
Let $\oneset$ be a compact subset of $\MGtb$, which is $\act$-invariant. 
Then $\act$ restricted to $\oneset$ is continuous and thus $(\oneset,\act)$ a dynamical system. We call it a \emph{translation bounded measures dynamical system (TBMDS)}.
For such a TBMDS $(\oneset,\act)$ we define the map 
$$N:\Cc\too\Lonetwo ~\text{ by }~
N(\fktG)(\onelem):=\onelem(\fktG):=\int_G\fktG (\g)\,\mathrm d\onelem$$
for $\fktG\in\Cc$.
Then $N(\fktG)$ is continuous on $\oneset$ by definition of the vague topology. 
Moreover, since $\oneset$ is a compact set, we obtain 
$$\absinf{N(\fktG)} = \sup_{\onelem\in\oneset}\abs{N(\fktG)}< \infty.$$
This yields that $N(\fktG)$ belongs to $\Loneinf$ for all $\fktG\in\Cc$ and is clearly compatible with translations.
Thus $\N=(N,\oneset,\mone,T)$ is a spatial process. It is also a full spatial process, as $N(\fktG)$ separates the points by Stone-Weierstrass.

\textbf{Remark.}
Note, that in literature a slightly different definition of the map $N$ can also be found. More details will be provided in the appendix.

The second class of examples concerns positive measures on $\G$ (see \cite{LeSt} for full details and proofs). 
Here we require  in addition that $\G$ has a countable basis of the topology. Let $\MG$ and $\act$ be defined as above and let $\Sigma_{\MG}$ denote the Borel $\sigma$-algebra on $\MG$. Since $\act$ is continuous in each variable, the action is also measurable.
An element $\nu$ of $\MG$ is called \emph{positive} if $\nu(\fktG)\geq 0$ for all $\fktG\in\Cc$ with $\fktG\geq 0$. 
We will consider $\MGp\subset\MG$, the set of \emph{positive measures on $\G$}, with the $\sigma$-algebra $\Sigma_{\MGp}$, which is the restriction of $\Sigma_{\MG}$ to $\MGp$. This set is invariant under the action $\act$ and thus $(\MGp, \act)$ is a measurable dynamical system on positive measures. 
Similar to the TBMDS, we define the map 
$$N:\Cc\too\{\text{functions on } \MGp\} ~\text{ by }~
N(\fktG)(\onelem):=\onelem(\fktG):=\int_G\fktG (\g)\,\mathrm d\onelem$$
for all $\fktG\in\Cc$.
Note that the functions $N(\fktG)$ are continuous and thus measurable and that $N(\fktG)\geq 0$ for $\fktG\geq 0$.
However $N(\fktG)$ is not bounded in general. Therefor we need an additional assumption: Let $m$ be an $\act$-invariant probability measure on $\MGp$, which is \emph{square integrable}, i.e. it satisfies 
$$\int_{\MGp}\left|N(\fktG)\right|^2\,\mathrm dm<\infty$$
for all $\fktG\in\Cc$.
Then $L^2(\MGp,m)$ is a Hilbert space and $\N=(N,\oneset,\mone,T)$ is a spatial process.
Since $\Sigma_{\MGp}$ is generated by the $N(\fktG)$ we obtain the fullness condition on $L^2(\MGp,m)$. Hence $\N=(N,\oneset,\mone,T)$ is a full spatial process and we call $(\MGp,\act)$ a \emph{positive measure dynamical system (PMDS)}.\\


\section{The equivalence of a spatial process to a TBMDS}
In this section we examine under which conditions a spatial process is equivalent to a TBMDS. We already know, that $N(\fktG)\in\Loneinf$ holds for all $\fktG\in\Cc$, in a TBMDS $\N=(N,\oneset,\mone,T)$. Our main concern will be the converse stated in the following theorem. 

\begin{theorem}\label{ThmEquiTBMDS}
Let $\N=(N,\oneset,\mone,T)$ be a full spatial process with $N(\fktG)\in L^{\infty}(\oneset,\mone)$ for all $\fktG\in \Cc$. 
Then $\N$ is equivalent to a translation bounded measure dynamical system (TBMDS).
\end{theorem}

The remainder of this section is devoted to the proof of Theorem \ref{ThmEquiTBMDS} with the actual proof given at the very end of this section.
The basic idea is to apply the Gelfand theory, see e.g. \cite{Ped}, to the commutative $C^*$-algebra  
$$\A:= \A(N(\Cc),\mathds{1}_{\A}).$$ 
This is the smallest algebra that is closed with respect to $\absinf{\cdot}$ and satisfies
\begin{itemize}
\item $N(\fktG)\in\A$ for all $\fktG\in\Cc$,
\item $\A$ has a unit $\mathds{1}_{\A}$.
\end{itemize}
The subalgebra $\A$ of $\Loneinf$ is obviously commutative,  $\norm{f\close{f}}=\norm{f}^2$ holds as well and it is complete by definition. 
By the Gelfand theory, we have 
$\A \iso C(\twoset)$ where $\twoset$ is a compact space. 
Specifically we obtain
$$\twoset:=\hat{\A},$$ 
where $\hat{\A}$ is the set of all nonzero, linear, multiplicative functionals on $\A$, which are automatically continuous by the Gelfand theory.
Note, that ${T_{\g}:\Lonetwo\too\Lonetwo}$ maps $\A$ into itself. Hence we can consider $T_{\g}$ as a map on $\A$ via $T_{\g}f=f\circ \act_{-\g}$ for $\g\in\G$.

The idea is to identify the elements of $\twoset$ with translation bounded measures by a suitable map 
$$\identif:\twoset\too\threeset\subset\MGtb.$$
In order to accomplish this, we have to show several statements. We start with the investigation of the algebra $\A$.

\begin{lemma}\label{AlgDense}
The algebra $\A$ is a dense subset of $\Lonetwo$. Furthermore the unit $\mathds{1}_{\A}$ is equal to $\mathds{1}_{\oneset}$, where $\mathds{1}_{\oneset}$ maps $\oneset$ to $\IC$ with $\onelem\mapsto 1$ for all $\onelem\in\oneset$.
\end{lemma}

\begin{proof}
We already know from the definition of spatial processes that
$$\B:=\lin\{\fktC_1(N(\fktG_1))\cdots\fktC_n(N(\fktG_n)):~\fktC_j\in\ \CC,~\fktG_j\in\Cc~\mbox{for}~j=1,\cdots,n,~n\in\IN\}$$ is dense in $\Lonetwo$. We show that $\B\subset\A$ and then infer that $\A$ is dense.

Let $\fktG_1,\cdots,\fktG_n\in\Cc$ be given  and $\fktC_1,\cdots,\fktC_n\in\CC$.
Since $N(\fktG_j)\in\Loneinf$, the image of $N(\fktG_j)$  is included in $B_{R_j}(0)$, which is the closed ball around $0$ with radius $R_j=\absinf{N(\fktG_j)}$. By the Stone–Weierstrass theorem we can approximate the compactly supported functions $\fktC_j$ uniformly by polynomials $p_k^{(j)}$ on $B_{R_j}(0)$ such that $p_k^{(j)}(N(\fktG_j))$ converges to $\fktC_j(N(\fktG_j))$ uniformly for $k\to\infty$ with respect to the norm on $L^{\infty}$. 

Now we set $P_k\defas p_k^{(1)}(N(\fktG_1))\cdots p_k^{(n)}(N(\fktG_n))$. Thus, $P_k$ is in $\A$ as $\A$ is an algebra. From the previous considerations, we obtain the convergence 
$$P_k\rightarrow\fktC_1(N(\fktG_1))\cdots\fktC_n(N(\fktG_n))\in\A$$ 
with respect to $\absinf{\cdot}$, since $\A$ is closed. 

As $m$ is a probability measure, convergence with respect to $\absinf{\cdot}$ implies convergence with respect to $\abstwo{\cdot}$. Thus we have $P_k\rightarrow\fktC_1(N(\fktG_1))\cdots\fktC_n(N(\fktG_n))$ in $L^2$. Consequently, we obtain  $\B\subset\A$ and $\A$ is dense in $\Lonetwo$.

Since $\mathds{1}_{\A}$ is the unit in $\A$
and $\A$ is dense in $\Lonetwo$, we obtain $\mathds{1}_{\A}=\mathds{1}_{\oneset}$.
\end{proof}

\begin{remark}
Note that we get the densness of $\A$ in $\Lonetwo$ from the fact that the $N(\fktG_j)$ are bounded so that the cutting off functions $\fktC_j$ have no effect in the end.
\end{remark}

\begin{prop}\label{xNcontin}
The map 
$$\twoelem\circ N:\Cc\too\IC,~\fktG\mapsto\twoelem(N(\fktG)),$$ 
is linear and continuous for all $\fktG\in\G$ and each $\twoelem\in\twoset$.
\end{prop}

\begin{proof}
The linearity follows directly from the linearity of $N$ and $\twoelem$. Since we already know that each $\twoelem\in \twoset$ is continuous it remains to show that $N$ is continuous as a map from $\Cc$ to $\A \sub \Loneinf$. This follows from the closed graph theorem and the continuity of $N$ as a map from $\Cc$ to $\Lonetwo$:

It suffices to show that the restriction of $N$ to  $\Cck$ is continuous for each compact subset $\subG$ of $\G$ (by definition of the topology on $\Cck$) and thus the Banach space $(\Cck,\absinf{\cdot})$. We will show that the map $N$ restricted to $\Cck$ is closed.

Let the sequences $(\fktGseq{n})_n\subset\Cck$ with $\fktG_n\too \fktG, n\to\infty$, and $N(\fktG_n)\too f, n\to\infty$, in $\Loneinf$ be given.\\ 
We already know that $N:\Cc \too \Lonetwo$ is a continuous map. This implies that $\Nrist:\Cck\too\Lonetwo$ is continuous as well and $N(\fktGseq{n})$ converges to $N(\fktG)$ in $\Lonetwo$. Thus there exists a sub-sequence $N(\fktGseq{n_k})$ with
$N(\fktGseq{n_k})\too N(\fktG), k\to\infty$, almost surely.\\
As $N(\fktG_n)$ converges to $f$ in $\Loneinf$, we also have the convergence of $N(\fktGseq{n_k})$ to $f$ in $\Loneinf$  and thus $f=N(\fktG)$, almost surly. Ergo, the map $\Nrist:\Cck\too\Loneinf$ is closed and by the closed graph theorem the statement follows.
\end{proof}

By the  Riesz-Markov  theorem, any continuous linear functional on $\Cc$ is induced by a measure. Thus, there exists to any 
$\twoelem\in\twoset$ a unique inner regular measure $\threelem$ on $G$ with 
$$\twoelem(N(\fktG))=\int_{G}\fktG~ d\threelem=:\threelem(\fktG)\text{ for all } \fktG\in\Cc$$ 
by the previous proposition.
We define 
$$\threeset\defas\{\threelem:\twoelem\in\twoset\}\subset\MG$$
equipped with the vague topology.
Furthermore, we consider the map $\identif$ given by 
$$\identif:\twoset\too\threeset,~\twoelem\mapsto\threelem ~\text{ and }~ \identif(\twoelem)(\fktG)=\threelem(\fktG)=\twoelem(N(\fktG))$$

holds. In the next proposition we will show some useful properties of these objects. In particular we will show that the $\threelem$ are translation bounded.

\begin{prop}\label{idenifIso}
The following holds:
\begin{enumerate}[(a)]
\item The measure $\threelem$ is translation bounded for all $\twoelem\in\twoset$.
\item The map $\identif$ is a homeomorphism and $\threeset$ is a compact space.
\item For $\twoelem\in\twoset$ and $\g\in\G$ the map $\twoelem_t:=\twoelem\circ T_{\g}: \A\too\IC$ belongs to $\twoset$ and\\ 
$$\threelem\circ\trans_{\g}=\threelemvari{\twoelem_t}$$
holds. In particular the space $\threeset$ is invariant under translation, i.e. $\threelem\circ\trans_{\g}$ belongs to $\threeset$ for all $\twoelem\in\twoset$ and $\g\in\G$.
\end{enumerate}
\end{prop}

\begin{proof}
\begin{enumerate}[(a)]
\item
By the definition of translation bounded measures it suffices to show that\\
$\threelem(\trans_{\g}\fktG)$ is bounded in $\g$ for all $\fktG\in\Cc$. By the continuity of $\twoelem$ and $N(\fktG)\in\Loneinf$, we obtain
$$
\abs{\threelem(\trans_{\g}\fktG)} 
= \abs{\twoelem(N(\trans_{\g}\fktG))}\\
\leq\|\twoelem\|\absinf{N(\trans_{\g}\fktG)}\\
= \|\twoelem\|\absinf{T_{\g}N(\fktG)},\\
= \|\twoelem\|\absinf{N(\fktG)},
$$
which is independent of $\g\in\G$.
\item We first show the continuity of $\identif$. For this purpose we have to show that $\threelemvari{\twoelem_i}$ converges to $\threelem$ for all nets $(\twoelem_i)_{i\in I}$ with $\twoelem_i\to\twoelem$. We have the weak-$*$-topology on $\Ahat$. This is the smallest topology such that all functionals 
$$(\a,\cdot):\Ahat\too\IC,~\ahat\mapsto(\a,\ahat):=\ahat(\a)$$ 
are continuous and therefore every sequence $(\a,\ahat_i)_{i\in I}$ converges to $(\a,\ahat)$ if $\ahat_i\to\ahat$ in $\Ahat$. 
Thus we obtain the convergence $\twoelem_i(\a)\to\twoelem(\a)$ for all $\a\in\A$. 
It follows that $\twoelem_i(N(\fktG))\to\twoelem(N(\fktG))$, as $\A$ is generated by the $N((\fktG)),~\fktG\in\Cc$ and $\charfkt{\A}$. Hence $\threelemvari{\twoelem_i}$ converges vaguely to $\threelem$.

Given the continuity of $\identif$, the compactness of $\threeset$ follows as 
$\twoset$ is compact.

By the definition of $\threeset$ the map $\identif$ is onto. It remains to show that $\identif$ is one-to-one. Let $\twoelem_1,\twoelem_2\in\twoset$ such that $\twoelem_1\neq \twoelem_2$. Then there exists a $\fktG\in\Cc$ with $\twoelem_1(N(\fktG))\neq\twoelem_2(N(\fktG))$. This implies $\threelemvari{\twoelem_1}(\fktG)\neq\threelemvari{\twoelem_2}(\fktG).$

\item The functional $\twoelem_t\defas\twoelem\circ T_{\g}$ is linear, multiplicative and continuous. The linearity and the continuity follow directly from the linearity and continuity of $T_{\g}$ and $\twoelem$. Regarding the multiplicativity, we have  
$$\twoelem\circ T_{\g}(fg)
=\twoelem((fg)\circ\act_{-\g})
=\twoelem(f\circ\act_{-\g})~\twoelem(g\circ\act_{-\g})
=\big(\twoelem\circ T_{\g}(f)\big)\big(\twoelem\circ T_{\g}(g)\big)$$
for all $f,g\in\A$.
Furthermore we have 
$$
\threelem(\trans_{\g}\fktG)
=\twoelem(N(\trans_{\g}\fktG))
=\twoelem\circ T_{\g}\circ N(\fktG)
=\twoelem_{\g}(N(\fktG))
=\threelemvari{\twoelem_t}(\fktG)
$$
\end{enumerate}
This finishes the proof.
\end{proof}

Now we can define the actions
$$\acthree: G\times\threeset\too\threeset ~\text{ of } G \text{ on } \threeset \text{ via }~\acthree_t(\threelem)(\fktG):=\threelem(\trans_{-\g}\fktG)$$
and
$$\actwo: G\times\twoset\too\twoset ~\text{ of } G \text{ on } \twoset \text{ via }~\actwo_{\g}(\twoelem)=:\twoelem\circ T_{-\g}=\threelemvari{\twoelem_t}$$
We derive
\begin{align*}
\acthree_{-\g}(\threelem)(\fktG)
=\threelemvari{\twoelem_t}(\fktG)
=\actwo_{-\g}(\twoelem)(N\fktG)
=\identif(\actwo_{-\g}(\twoelem))(\fktG).
\tag{$\blacktriangle$}
\end{align*}
At this point we have the identification of $\twoset$ with the set $\threeset$ whose elements are translation bounded measures. The next step is to define the measure $\mthree$ and a map 
$${\Nthree:\Cc\too\Lthree},$$ 
which satisfies the desired properties, such that $(\Nthree,\threeset,\mthree,\Three)$ is a spatial process. 
Here $\Three$ is the Koopman representation on $\Lthree$ induced by $\acthree$.

We already have the map 
$$\Nthree:\Cc\too C(\threeset)\text{ given by }\Nthree(\fktG)(\threelem)\asdef\threelem(\fktG)$$ 
for $\threelem\in\threeset$, 
which maps indeed into $C(\threeset)$ by the definition of the vague topology.
Furthermore $\Nthree$ is compatible with complex conjugation. This follows by the Gelfand theory and the properties of $N$:
\begin{align*}
\close{\Nthree(\fktG)(\threelem)}
&=\close{\threelem(\fktG)}\\
&=\close{\twoelem( N(\fktG))}\\
(\text{Gelfand theory})~&=\twoelem(\close{N(\fktG)})\\
(\text{property of $N$})~&=\twoelem(N(\close{\fktG}))\\
&=\threelem(\close{\fktG})\\
&=\Nthree(\close{\fktG})(\threelem)
\end{align*}

To obtain the measure $\mthree$ we use the Gelfand theory of C*-algebras again, more precisely the Gelfand representation.
This representation is the isometric-*-isomorphism 
$\gelf:\A\too C(\twoset)$ with $\gelf(f)(\twoelem)=\twoelem(f)$
as $\twoset=\Ahat$
and induces a measure $\mtwo$ on $\twoset$ via 
$$\int_{\twoset}\ftwo\,d\mtwo=\int_{\oneset}\gelf^{-1}(\ftwo)\,d\mone~\mbox{ for }~ \ftwo\in C(\twoset).$$
The measure $\mtwo$ is a probability measure as 
$$
\mtwo(\twoset)
=\int_{\twoset} \I_{\twoset}\,d\mtwo
=\int_{\oneset} \gelf^{-1}(\I_{\twoset})\,d\mone
=\int_{\oneset} \I_{\oneset} \,d\mone
=\mone(\oneset).
$$
shows.
Since the algebra $\A$ is dense in $\Lonetwo$ the Gelfand representation
can be extended to a unitary operator $U:\Lonetwo\too\Ltwo$ with $U(\I_{\oneset})=\I_{\twoset}$ and we obtain for $\ftwo\in\Ltwo$
$$\int_{\twoset} \ftwo\,d\mtwo
=\langle f, \I_{\twoset} \rangle
=\langle U^{-1}(f), U^{-1}(\I_{\twoset})\rangle
=\langle U^{-1}(f), \I_{\oneset}\rangle
=\int_{\oneset} U^{-1}(\ftwo)\,d\mone.$$
In a similar way, the measure $\mthree$ on $\threeset$ is induced by the homeomorphism $\identif:\twoset\too\threeset$ via
$$\int_{\threeset} f\,d\mthree
=\int_{\twoset} f\circ\identif\,d\mtwo~\mbox{ for }~ f\in L^{1}(\threeset,\mthree).$$
Thus, we find
$$\int_{\threeset}f\,d\mthree
=\int_{\oneset}U^{-1}(\fthree\circ\identif)\,d\mone~\mbox{ holds for }~ \fthree\in\Lthree.$$
In particular we have
$$
\mthree(\threeset)
=\int_{\threeset} \I_{\threeset}\,d\mthree
=\int_{\oneset} U^{-1}(\I_{\threeset}\circ\identif)\,d\mone
=\int_{\oneset} U^{-1}(\I_{\twoset})\,d\mone
=\int_{\oneset} \I_{\oneset} \,d\mone
=1.
$$
and $\mthree$ is a probability measure as well.

It remains to show that the induced measure $\mthree$ is a translation invariant probability measure, whence $C(\threeset)\subset\Lthree$ and $\Nthree:\Cc\too \Lthree$ follows.

\begin{prop}\label{invMeas}
The measure $\mtwo$ is invariant on $\twoset$ under the action $\actwo$ and the measure $\mthree$ is invariant on $\threeset$ under the action $\acthree$.
\end{prop}

\begin{proof}
We use the invariance of the measure $\mone$ on $\oneset$ and the fact that the measure $\mtwo$ is induced by the Gelfand representation $\gelf$ to show the invariance of $\mtwo$ on $\twoset$. 

Let $f\in \A$ and $f'=\gelf(f)\in C(\twoset)$. From the properties of $\gelf$ and the definition of $\actwo $ we obtain 
$$
\gelf(T_{-\g}f)(\twoelem)
= \twoelem(T_{-\g}f)
=\actwo_{\g}(\twoelem)(f)
= \gelf(f)(\actwo_{\g}(\twoelem))
= f'(\actwo_{\g}(\twoelem))
= f'\circ\actwo_{\g} (\twoelem)
$$
%
%
Hence we have
\begin{align*}
\int_{\twoset}(f'\circ\actwo_{\g})(\twoelem)\,d\mtwo
&=\int_{\twoset}\gelf(T_{-\g}f)(\twoelem)\,d\mtwo\\
&=\int_{\oneset}\gelf^{-1}(\gelf(T_{-\g}f))(\twoelem)\,d\mone\\
&=\int_{\oneset}T_{-\g}(f(\onelem))\,d\mone\\
&=\int_{\oneset}f(\onelem)\,d\mone\\
&=\int_{\twoset}\gelf (f)(\twoelem))\,d\mtwo\\
&=\int_{\twoset}f'(\twoelem)\,d\mtwo
\tag{$\blacktriangledown$}
\end{align*}

It follows for $f\in C(\threeset)$ that
$$
\int_{\threeset}f\circ\acthree_{\g}\,d\mthree\\
=\int_{\twoset}f\circ\acthree_{\g}\circ \identif\,d\mtwo\\
\overset{(\blacktriangle)}=\int_{\twoset}f\circ \identif\circ\actwo_{\g}\,d\mtwo\\
\overset{(\blacktriangledown)}=\int_{\twoset}f\circ \identif\,d\mtwo\\
=\int_{\threeset}f\,d\mthree\\
$$
and thus the invariance of the measure $\mthree$.
\end{proof}

By the invariance of the measures $\mtwo$ and $\mthree$ we can define the Koopman operators $\Two_{\g}$ and $\Three$ induced by the actions $\actwo$ respectively $\acthree$.
With the consideration on the actions $\actwo$ and $\acthree$ from above we derive
$$\Nthree(\trans_{\g}\fktG)(\threelem)
=\threelem(\trans_{\g}\fktG)
=\acthree_{-\g}(\threelem)(\fktG)
=\Nthree(\fktG)(\acthree_{-\g}(\threelem))
=\Three_{\g}(\Nthree(\fktG))(\threelem).$$
 
So far we have shown that $\N_{\identif}=(\Nthree,\threeset,\mthree,\Three)$ is a spatial process of translation bounded measures and hence a TBMDS.
The last thing we have to do is to construct a unitary map $\unit:\Lonetwo\too\Lthree$ such that $\unit$ is a spatial isomorphism between $\N$ and $\N_{\identif}$.

We start with the following statements which apply to the Koopman operators. 
%
%
%
\begin{prop}\label{commutU}
For the unitary operator $U$ and the Koopman operators $\Two_{\g}$ and $\Three$ we have for all $\g\in\G$ the following:
\begin{enumerate}[(a)]
\item $U\circ T_{\g}=\Two_{\g}\circ U$.

\item $(\Three_{\g}\fthree)\circ\identif=\Two_{\g}(f\circ\identif)$. 
\end{enumerate}
\end{prop}

\begin{proof} (a)\,
We start to show that $\gelf$ commutes with the Koopman operator $T_{\g}$.
\begin{align*}
\gelf(T_{\g}(N(\fktG)))(\twoelem))
=\twoelem(T_{\g}(N(\fktG))
=\actwo_{-\g}(\twoelem)(N(\fktG))
=\gelf(N(\fktG))(\actwo_{-\g}(\twoelem))
=\Two_{\g}(\gelf(N(\fktG)))(\twoelem).
\end{align*}

Note that $\gelf^{-1}$ commutes with $\Two_{\g}$ as well, since $\gelf$ is an algebren isomorphism. 
Hence, the same is true for $U$ and $U^{-1}$, as $\gelf$ is the restriction of $U$  to $\A$.

(b)\, For the operators $\Two_{\g}$ and $\Three_{\g}$ we have
$$(\Three_{\g}\fthree)(\threelem)
=\fthree\circ\acthree_{-\g}(\threelem)
\overset{(\blacktriangle)}=(\fthree\circ\identif)(\actwo_{-\g}(\twoelem))
=\Two_{\g}(\fthree\circ\identif)(\twoelem)
$$
which finishes the proof.
\end{proof}

We continue with some considerations concerning the map
$$\Identif:\Lthree\too\Ltwo\text{ defined as }f\mapsto f\circ\identif.$$

\begin{prop}\label{eigIdent} 
The following holds:
\begin{enumerate}[(a)]
\item For any $\fktG\in\Cc$, we have $\Identif(\Nthree(\fktG))= \gelf(N(\fktG))$.
\item $\Identif$ is a unitary map.\\
\end{enumerate}
\end{prop}

\begin{proof}(a)\,
A direct computation gives
$$\Identif(\Nthree(\fktG))(\twoelem)
=(\Nthree(\fktG)\circ \identif)(\twoelem)
=\Nthree(\fktG)(\threelem)
=\threelem(\fktG)=\twoelem(N(\fktG))
=\gelf(N(\fktG))(\twoelem)$$ for $\twoelem\in\twoset$.

(b)\, In Proposition \ref{idenifIso} we have already shown that $\identif:\twoset\too\threeset,~\twoelem\mapsto\threelem$, is a homeomorphism, thus the statement is clear from the definition of $\mthree$.
\end{proof}

We set
$$\Uthree\defas\Identif^{-1}\circ\Unit:\Lonetwo\too\Lthree$$
and obtain $\Uthree^{-1}\defas \Unit^{-1}\circ\Identif:\Lthree\too\Lonetwo$.
Now we have to check that 
$U_{\identif}$ satisfies the properties mentioned in Definition \ref{EquiProc}. This is straightforward. The details are provided in the next lemma.

\begin{lemma}\label{SpatialIso}
Let $(N,\oneset,\mone,T)$ be the spatial process from Theorem \ref{ThmEquiTBMDS}. Let $(\Nthree,\threeset,\mthree,\Three)$ be the spatial process over a set of translation bounded measures constructed above. Then $\Uthree$ is a unitary operator between $\Lonetwo$ and $\Lthree$ which satisfies the following conditions: 
\begin{itemize}
\item $\Uthree\circ N=\Nthree$.
\item $\Uthree \circ T_{\g}= \Three_{\g}\circ\Unit$ for all $\g\in G$.
\item $\Uthree(fg)=\Uthree(f)\Uthree(g)$ as well as $\Uthree^{-1}(f',g')=\Uthree^{-1}(f')\Uthree^{-1}(g')$\\
for all $f,g\in\Loneinf$ and all $f',g'\in \Lthreeinf$.
\end{itemize}
Thus $(N,\oneset,\mone,T)$ and $(\Nthree,\threeset,\mthree,\Three)$ are equivalent.
\end{lemma}

\begin{proof}
The maps $\Uthree:\Lonetwo\too\Lthree$ and $\Uthree^{-1}:\Lthree\too\Lonetwo$ are obviously well-defined and unitary, since $\Unit$ and $\Identif$ are unitary.

\begin{enumerate}[$\bullet$]
\item For the first point we have $\Identif\circ\Nthree=\Unit\circ N$ by Proposition \ref{eigIdent}. This yields $\Unit^{-1}\circ\Identif\circ\Nthree=N$ and $\Uthree\circ N=\Nthree$.

\item For the second point we use $(\Three_{\g}\fthree)\circ\identif=\Two_{\g}(\fthree\circ\identif)$, shown in Proposition \ref{commutU} and obtain $\Identif\circ\Three_{\g}=\Two_{\g}$, from the definition of $\Identif$. This is equivalent to $\Three_{\g}=\Identif^{-1}\circ\Two_{\g}$ and yields 
$$\Uthree \circ T_{\g}=\Identif^{-1}\circ\Unit\circ T_{\g}=\Identif^{-1} \circ\Two_{\g}\circ\Unit=\Three_{\g}\circ\Unit.$$

\item The last point follows directly from the properties of $\Unit$ and $\Identif$.\\ The operator $\Unit$ is multiplicative since it is the extension of the Gelfnad representation and for the map $\Identif$ we have
$$
\Identif(fg)
=(fg)\circ \identif\\
=(f\circ \identif)(g\circ\identif)\\
=\Identif(f)\Identif(g).
$$
\end{enumerate}
This finishes the proof
\end{proof}

\begin{proof}[Proof of Theorem \ref{ThmEquiTBMDS}]
This is a direct consequence of Lemma \ref{SpatialIso}.
\end{proof}



\section{The equivalence between spatial processes and PMDS}
Analogously to the second section this section deals with the equivalence of spatial processes to  PMDS. For PMDS we obtain $N(\fktG)\geq 0$ for $\fktG\in\CC$ with $\fktG\geq 0$ from the definition. In the following we will show the converse as stated in Theorem \ref{equiPMDS}.
Since we are focused on positive functions we use
$$\Cc^+ := \{ \fktG \in\Cc \text{ with }\fktG\geq 0\}$$
throughout this section. 

\begin{theorem}\label{equiPMDS}
Let $\N=(N,\oneset,\mone,T)$ be a full spatial process.
Then $\N$ is equivalent to a positive measure dynamical system (PMDS ) if 
$$N(\fktG)\geq0 ~\mbox{ holds for all }~ \fktG\in\Cc^+.$$
\end{theorem}

To prove this theorem, we take a similar approach as above, again using the theory of C*-algebras to identify $\N$ with a process over positive measures.
Since the end of the proof is analogical to the proof of Theorem \ref{ThmEquiTBMDS}
we will only show the beginning here and leave the rest to the reader.

As already discussed in \cite{LeMo}  the fullness condition is in some sense a convention. However, for our purpose we need to extend the definition involving continuous functions with compact support to a definition involving uniformly continuous bounded functions but the rest of the theory remains valid. We denote the space of uniformly continuous bounded functions from $\IC$ to $\IC$ by $\Cu$ and set for this section 
$${\A\defas\lin\{\fktC_1(N(\fktG_1))\cdots\fktC_n(N(\fktG_n)):\,n\in\IN,\, \fktC_j\in \Cu,\,\fktG_j\in\Cc\mbox{ for }j=1,\cdots,\, n\}}\cup\{\mathds{1}_{\A}\}.$$
The process $\N$ is then a full spatial process with respect to $\A$, as $\A$ is a dense subset of $\Lonetwo$. This can be deduced from the proof of Proposition 5.2 in \cite{LeMo} and the fact that the composition $h_n\circ\fktC$ converges to $h\circ\fktC$ for every sequence of functions $h_n\in\Lonetwo$ with $h_n\too h,~n\to\infty$ and $\fktC\in\Cu$. The latter will be shown in the appendix.

Note that the statement of Proposition 5.2 as well as Theorem 5.4. in \cite{LeMo}, refer to bounded continuous functions. However, the proof of Proposition 5.2, referring to Proposition 4.10, only works for continuous functions with compact support. In this sense the statements of Proposition 5.2 and Theorem 5.4. are not proven as stated.

The Gelfand representation $\gelf$ maps $\A$ to $C(\twoset)$ and $C(\twoset)$ is a dense subset of $\Ltwo$, as $X=\Ahat$ is compact. This property allows us to extend the map $\gelf$ to a unitary operator 
$$U:\Lonetwo\too\Ltwo ~\text{ and we set }~ N':= U\circ N.$$ 
Note that $U$ coincide with the Gelfand representation $\gelf$ on $\A$ and inherits the properties of $\gelf$, as its extension, but in contrast to $\gelf$, the operator $U$ is defined on the whole space $\Lonetwo$.

As we did it for  TBMDS, the purpose is to identify the elements of $\twoset$ with positive measures but there is a problem:
Now $N(\fktG)$ only belongs to $\Lonetwo$ and not necessarily to $\A$. In particular $N(\fktG)$ does not need to be a bounded function and it is only defined almost everywhere. Hence the evaluation at $\twoelem \in\Ahat$ is not well defined.

To deal with that problem, we consider a sequence $\fktCseq{n}\in\Cu$ with 
 $$\fktCseq{n}(z):=\begin{cases}
  0:  & \text{for }z<0\\
  z: & \text{for }0\leq z\leq n,
\end{cases} ~\mbox{ and }~\fktCseq{n}\leq\fktCseq{n+1} ~\mbox{ for all }~n\in\IN.$$
Then $\fktCseq{n}$ converges pointwise monotonously increasing to the identity on $\IR^+$. (Note that we consider only $\fktG\in\CC^+$, so that the convergence on $\IR^+$ is sufficient for our purpose.) Furthermore, we choose a sequence $(\fktGseq{n})_{n\in\IN}\in\Cc$ such that
$$0\leq\fktGseq{n}\leq\fktGseq{n+1}\leq 1 ~\mbox{ and }~\fktGseq{n}=1 \mbox{ on a compact }\subG_n\subset\G$$
with $\subG_n^{\circ}\subset\subG_{n+1}$ for all $n\in\IN$ and $\bigcup_{n\in\IN}\subG_n=\G$. With these two sequences we obtain the following result.

\begin{lemma}\label{exGW}
Let $(\fktCseq{n})_{n\in\IN}$ and $(\fktGseq{n})_{n\in\IN}$ be defined as above and let $\fix\in\IN$ be fixed. Then the sequence $\gelf(\fktCseq{n}(N(\fktGseq{\fix})))(\twoelem)$ is pointwise increasing for all $\twoelem \in \twoset$. The limit
$$\Nrx(\fktGseq{\fix})(\twoelem)\defas\lim_{n\in\IN}\gelf(\fktCseq{n}(N(\fktGseq{\fix})))(\twoelem)$$ is almost everywhere finite and agrees almost everywhere with $N'(\fktGseq{\fix})$. 
\end{lemma}

\begin{proof}
The sequence $(\fktGseq{n})_{n\in\IN}$ is a subset of $\Cc^+$  by definition. Thus, $N(\fktGseq{n})(\onelem)\geq 0$ holds for all $n\in\IN$, as $N$ is positive. 
For this reason, we can apply $\fktCseq{n}$ to $N(\fktGseq{\fix})$. By $\fktCseq{n}\leq \fktCseq{n+1}$, we receive the monotonously increasing sequence $\big(\fktCseq{n}(N(\fktGseq{\fix}))\big)_{n\in\IN}\subset\A$, which converges to $N(\fktGseq{\fix})$ in $\Lonetwo$. Hence the sequence $U(\fktCseq{n}(N(\fktGseq{\fix})))$ converges to $N'(\fktGseq{\fix})$ in $\Ltwo$.
Since the Gelfand representation $\gelf$ is positive as well, the sequence
$$
\gelf(\fktCseq{n}(N(\fktGseq{\fix})))\in\C{\twoset}\subset\Ltwo,~n\in\IN
$$
is continuous, monotonously increasing and well defined for every $\twoelem\in\twoset$. As $U$ coincides with $\gelf$ on $\A$ the sequence $\Unit(\fktCseq{n}(N(\fktGseq{\fix})))$ is also well defined for every $\twoelem\in\twoset$ and converges pointwise in $\Ltwo$ to an element $\Nrx(\fktGseq{\fix})$. 
The map $\Nrx(\fktGseq{\fix})$ is a representative of $\Nx(\fktGseq{\fix})\in\Ltwo$. Hence the limit defined above is finite almost everywhere.
\end{proof}

\begin{remark}
Note that the pointwise limit in Lemma \ref{exGW} is independent of the choice of the sequence $\fktCseq{n}$. This can easily be verified by considering another sequence $\fktCSeq{n}\in\cc{\IR}$ with the same properties. Then there exists for every $n\in\IN$ an element $m\in\IN$ and for every $m\in\IN$ an element $k\in\IN$ such that $\fktCseq{n}\leq\fktCSeq{m}\leq\fktCseq{k}$ holds. Consequently the limits are equal.
\end{remark}

Next we will show that this limit is finite for all $\fktG\in\Cc^+$ and all $\twoelem\in\twoset$ up to a set of measure zero, which does not depend on $\fktG$. 
\begin{prop}
There exists a null-set $\nset\subset \twoset$, with respect to the measure $\mtwo$ such that the limit 
$$\Nrx(\fktG)(\twoelem)\defas\lim_{n\in\IN}\gelf(\fktCseq{n}(N(\fktG)))(\twoelem)$$
is finite for all $\fktG\in\Cc^+$ and all $\twoelem\in\twoset\setminus\nset$.  
\end{prop}

\begin{proof}
For all $\fktG\in\Cc^+$ there  exist $k,\fix\in\IN$ satisfying
$\fktG\leq k\fktGseq{\fix}$ with ${\supp(\fktG)\subset\subG_L}$ and $\absinf{\fktG}\leq k$.
By the same arguments that provided the existence of the limit in Lemma \ref{exGW}, we obtain that
$$\Nrx (k\fktGseq{\fix})(\twoelem)\defas\lim_{n\in\IN}\gelf(\fktCseq{n}(N (k\fktGseq{\fix})))(\twoelem)$$
exists pointwise for all $\twoelem\in\twoset$ and is finite almost everywhere.
We set 
$$\nset_{k,\fix}\defas\{\twoelem \in \twoset:~\Nrx (k\fktGseq{\fix})(\twoelem)=\infty\},$$
then $\nset_{k,\fix}$ is a null-set for every $k\in\IN$ and every $\fix\in\IN$.
Thus 
$$\nset\defas\bigcup_{\fix\in\IN}\bigcup_{k\in\IN}\nset_{k,\fix}$$ is a null-set as well (as a countable union of null-sets).
By the assumption that $N$ is positive we have $N(\fktG)\leq N (k\fktGseq{\fix})$, since $N$ is linear, and $\Nrx(\fktG)\leq\Nrx (k\fktGseq{\fix})$ follows from the definition of $\fktCseq{n}$. This finishes the proof.  
\end{proof}

Now we can apply the result to an arbitrary $\fktG\in\Cc$, by decomposing $\fktG\in\Cc$ in $\fktG=\fktG_{r_+}-\fktG_{r_-}+i\fktG_{c_+}-i\fktG_{c_-}$ with $\fktG_{r_+},\fktG_{r_-},\fktG_{c_+},\fktG_{c_-}\in\Cc^+$.

We set $\twoset\setminus\nset\asdef\twoset'$ and obtain that 
\begin{align}\label{N^'}
\Nrx(\fktG)\defas\Nrx(\fktG_{r_+})-\Nrx(\fktG_{r_-})+i\Nrx(\fktG_{c_+})-i\Nrx(\fktG_{c_-})
\tag{$\star$}
\end{align}
is finite for all $\fktG\in\Cc$ and all $\twoelem\in\twoset'$. It remains to show that $\Nrx$ is a linear map. 
For this purpose we consider a certain sequence of functions $\fktCSeq{n}\in\Cu$, defined as follows
$$\fktCSeq{n}(z)=:\begin{cases}
  0:  & \text{for }z<0\\
  z: & \text{for }0\leq z\leq n\\
  n: & \text{for }n\leq z,
\end{cases}  ~\mbox{ for all }~n\in\IN.$$

\begin{prop}
For each $\twoelem\in\twoset'$ the map
$\Cc\too\IC,\, \fktG\mapsto \Nrx(\fktG)(\twoelem)$ is linear, positive and continuous.
\end{prop}
\begin{proof}
We show the linearity for the positive functions $\fktG\in\Cc^+$. Then the statement follows easily for all $\fktG\in\Cc$ from Equation 
(\ref{N^'}).

We start by showing two inequalities concerning the sequence $\fktCSeq{n}$.
To verify these inequalities we will only consider non negative numbers $s,t\in\IR$ and differ for each inequality four cases.
The first inequality is
$$\fktCSeq{n}(s+t)\leq \fktCSeq{n}(s)+\fktCSeq{n}(t):$$
Case 1: For $s+t \leq n$ we have $s,t\leq n$ and 
$\fktCSeq{n}(s+t) = s+t =  \fktCSeq{n}(s)+\fktCSeq{n}(t)$.\\
Case 2: For $s+t > n$ with $s\leq n$ and $t\leq n$ we have 
$\fktCSeq{n}(s+t) = n < s+t = \fktCSeq{n}(s)+\fktCSeq{n}(t)$.\\
Case 3: For $s+t > n$ with $s\leq n$ and $t> n$ we have 
$\fktCSeq{n}(s+t) = n \leq s+n =  \fktCSeq{n}(s)+\fktCSeq{n}(t)$.\\
Case 4: For $s+t > n$ with $s> n$ and $t> n$ we have 
$\fktCSeq{n}(s+t) = n < 2n = \fktCSeq{n}(s)+\fktCSeq{n}(t)$.

Next we consider the second inequality 
$$\fktCSeq{n}(s)+\fktCSeq{n}(t) \leq \fktCSeq{2n}(s+t):$$

Case 1: For $s,t\leq n$ we have $s+t \leq 2n$ and obtain
$\fktCSeq{n}(s)+\fktCSeq{n}(t) =s+t \leq \fktCSeq{2n}(s+t)$.\\
Case 2: For $s+t \leq 2n$ with $s < n$ and $t > n$ we have 
$\fktCSeq{n}(s)+\fktCSeq{n}(t) = s+n < s+t =\fktCSeq{2n}(s+t)$.\\
Case 3: For $s+t > 2n$ with $s\leq n$ and $t> n$ we have 
$\fktCSeq{n}(s)+\fktCSeq{n}(t) = s+n \leq 2n = \fktCSeq{2n}(s+t)$.\\
Case 4: For $s+t > 2n$ with $s> n$ and $t> n$ we have 
$\fktCSeq{n}(s)+\fktCSeq{n}(t)=2n=\fktCSeq{2n}(s+t)$.

With $\fktG\in\Cc^+$ and the additivity of $N$ follows 
$$
\fktCSeq{n}(N(\fktG+\psi))\leq \fktCSeq{n}(N(\fktG)) +\fktCSeq{n}(N(\psi))
~\text{ and }~ \fktCSeq{n}(N(\fktG)) + \fktCSeq{n}(N(\psi)) \leq \fktCSeq{2n}(N(\fktG+\psi)).
$$
Thus 
$$\lim_{n\in\IN}\fktCSeq{n}(N(\fktG+\psi)) = \lim_{n\in\IN}\fktCSeq{n}(N(\fktG))+
\lim_{n\in\IN}\fktCSeq{n}(N(\psi)) = \lim_{n\in\IN}\fktCSeq{2n}(N(\fktG+\psi))$$
holds. We deduce 
$$\lim_{n\in\IN}\gelf(\fktCSeq{n}(N(\fktG+\psi))) = \lim_{n\in\IN}\gelf(\fktCSeq{n}(N(\fktG)))+
\lim_{n\in\IN}\gelf(\fktCSeq{n}(N(\psi))) = \lim_{n\in\IN}\gelf(\fktCSeq{2n}(N(\fktG+\psi)))$$
(as $\gelf$ is linear and continuous) and obtain 
$$\Nrx(\fktG+\psi)=\Nrx(\fktG)+\Nrx(\psi).$$

To show the multiplicativity of $\Nrx$ we deduce $\Nrx(k\fktG)=k\Nrx(\fktG)$ from the additivity of $\Nrx$ for all $k\in\IN$ and thereby
$k \Nrx(\frac{1}{k}\fktG)=\Nrx(\frac{k}{k}\fktG)=\Nrx(\fktG)$. Thus we have 
$$\frac{l}{k}\Nrx(\fktG)=\frac{1}{k}\Nrx(l\fktG)=\Nrx(\frac{l}{k}\fktG),$$
which gives us 
$$\Nrx(\lambda\fktG)=\lambda\Nrx(\fktG)$$
for all $\lambda\in\IQ$.
Finally we obtain the compatibility with scalar multiplication of $\Nrx$ by the approximation of elements in $\IR$ via elements in $\IQ$ and the continuity of $\gelf$.
Obviously the function $\Nrx$ is positive since $N,\fktCseq{n}$ and $\gelf$ are positive and $\fktCseq{n}$ is monotonously increasing. Thus we obtain the continuity as well, as positive linear functionals are continuous.
\end{proof}

At this point, we have that the map $\fktG\mapsto\Nrx(\fktG)(\twoelem)$ is bounded, positive and linear for each $\twoelem\in\twoset'$ and
all $\fktG\in\Cc$.
Thus $\Nrx$ is a positive continuous functional.
Now we can apply the Riesz-Markov theorem and obtain for each $\twoelem\in\twoset'$ the existence of a unique positive measure $\threelem$ satisfying 
$$\Nrx(\fktG)(\twoelem)=\int\fktG\,d\threelem\defas\threelem(\fktG).$$

Like in the case of TBMDS, we set again
$$\threeset\defas\{\threelem:\twoelem\in\twoset'\} ~\mbox{ and }~
\identif:\twoset'\too\threeset,~\twoelem\mapsto\threelem.$$
The invariance under translation follows similarly to the proof of Proposition \ref{idenifIso}. Thus we have the identification of $\twoset'$ with $\threeset$.

Next we define the map $N':\Cc\too\Ltwo$ via 
$$N'(\fktG):=\left[\Nrx(\fktG)\right],$$
where $\left[\Nrx(\fktG)\right]$ denotes the equivalence class of $\Nrx(\fktG)$ in $\Ltwo.$
Clearly $\Nrx(\fktG)= f$ almost everywhere for all $f\in N'(\fktG)$ by definition. 
Finally we want to make sure that the extension $\Unit$ of the Gelfand representation $\gelf$ satisfies the necessary condition for the equivalence of spatial processes.

\begin{prop}
Let ${\Unit :\Lonetwo\too\Ltwo}$ be the extension of $\gelf$ to $\Lonetwo$. Then
$\Nrx(\fktG)\in \Unit(N(\fktG))$
holds for all $\fktG\in\Cc$ so that
$$N'=\Unit\circ N.$$
\end{prop}
\begin{proof}
Since $\Unit$ coincides with $\gelf$ on $\A$ and inherits its properties as its extension it follows that 
$$\lim_{n\in\IN}\gelf(\fktCseq{n}(N(\fktG)))(\twoelem)
=\lim_{n\in\IN}\Unit(\fktCseq{n}(N(\fktG)))(\twoelem)$$ 
in $\Ltwo$ and is finite almost everywhere for all $\fktG\in\Cc$ as well as positive and linear.
Since $\Unit$ is unitary we obtain
$$\lim_{n\in\IN}\Unit(\fktCseq{n}(N(\fktG)))(\twoelem)
=\Unit(\lim_{n\in\IN}\fktCseq{n}(N(\fktG)))(\twoelem)
=\Unit(N(\fktG))(\twoelem).$$
Thus $\Nrx(\fktG)(\twoelem)\in\Unit(N(\fktG))(\twoelem)$ holds for all $\fktG\in\Cc$, which implies 
$N'=\Unit\circ N$.
\end{proof}

Now we can continue in the same way as we did for the TBMDS. We will not go into details since the proofs are analogus to the case already considered.


\section{Appendix}
The appendix has two parts. The first part deals with a slightly different definition of the function $N:\Cc\too\LtwoG$ in the
articles \cite{LeMo} and \cite{BaLe04}.
Indeed these two definitions appear in many places in the literature. However, they lead to the same autocorrelation measure $\gamma$.
This is discussed below. 

The convolution of two functions $\fktG,\psi\in\CC$ defined as $\fktG * \psi\in\CC$ is given by 
$$(\psi*\fktG)(\g)=(\fktG * \psi)(\g):=\int_{\G}\fktG(\g-s)\psi(s)\,\mathrm ds =\int_{\G}\fktG(s)\psi(\g-s)\,\mathrm ds.$$

The convolution of a measure $\gamma$ with a function $\fktG\in\CC$ is given by
$$(\fktG*\gamma)(\g):=\int_{\G}\fktG(\g-s)\,\mathrm d\gamma(s)$$
 and
$\tilde{\fktG}(\g):=\overline{\fktG}(-\g)$ for $\fktG\in\CC$.
Let $\N=(N,\oneset,\mone,T)$ be a spatial process,more specifically even a TBMDS, as defined in  the preliminaries. 

In \cite{LeMo}
we have $\langle f,g \rangle_1 :=\int_{\twoset}f\overline{g}\,\mathrm dm$ for the inner product of $f,g\in\LtwoG$.
The linear function  $N:=N_1:\Cc\too\LtwoG$ is given via
$$N_1(\fktG)(\onelem):=\int_G\fktG (\g)\,\mathrm d\onelem $$
and the autocorrelation measure $\gamma_1$ is unique with
$$\gamma_1(\fktG*\tilde{\psi}):= \langle N_1(\fktG),N_1(\psi) \rangle.$$

In \cite{BaLe04} the inner product in $\LtwoG$ is given by $\langle f,g \rangle_2 :=\int_{\twoset}\overline{f}g\,\mathrm dm$. 
The function $N:=N_2:\Cc\too\LtwoG$ is now defined as 
$$N_2(\fktG)(\onelem):=\int_G\fktG (-\g)\,\mathrm d\onelem $$
and the autocorrelation measure $\gamma_2$ is unique with
$$(\tilde{\fktG}*\psi*\gamma_2)(0):= \langle N_2(\fktG),N_2(\psi) \rangle.$$

\begin{prop}
Let everything be given as discussed above, then
$\gamma_1=\gamma_2$
holds.
\end{prop}

\begin{proof}
By these two definitions of the autocorrelation measure we find
 $\langle f,g \rangle_1 = \langle g,f \rangle_2$ for $f,g\in\LtwoG$
and $N_1(\fktG)(\onelem) = N_2(\fktG_-)(\onelem)$ for $\fktG\in\CC$ with $\fktG_-(\g):=\fktG(-\g)$
Thus we obtain $\langle N_1(\fktG),N_1(\psi) \rangle_1 = \langle N_2(\psi_-), N_2(\fktG_-)\rangle_2.$\\

We first show that $(\fktG_-* \psi_-)(\g)=(\fktG* \psi)_-(\g)$:
\begin{align*}
(\fktG_-* \psi_-)(\g)
&= \int_{\G}\fktG_-(s)\psi_-(\g-s)\,\mathrm ds\\
&= \int_{\G}\fktG(-s)\psi(s-\g)\,\mathrm ds\\
&= \int_{\G}\fktG(s)\psi(-\g-s)\,\mathrm ds\\
&= (\fktG*\psi)(-t)\\
\end{align*}

Putting all of this together and replacing $\psi$ by $\tilde{\psi}$ we obtain
\begin{align*}
\gamma_1(\fktG*\tilde{\psi})
&= \langle N_1(\fktG),N_1(\psi) \rangle_1\\ 
&= \langle N_2(\psi_-), N_2(\fktG_-)\rangle_2\\
&= (\tilde{\psi_-}*\fktG_-*\gamma_2)(0)\\
&= (\tilde{\psi}*\fktG)_-*\gamma_2(0)\\
&= \int_G(\fktG*\tilde{\psi})_-(-\g)\,\mathrm d\gamma_2(\g)\\
&=\int_G(\fktG*\tilde{\psi})(\g)\,\mathrm d\gamma_2(\g)\\
&= \gamma_2(\fktG*\tilde{\psi}).
\end{align*}
Hence it is $\gamma_1=\gamma_2$, as $\fktG*\psi$ for $\fktG,\psi\in\CC$ is dense in $\CC$.
\end{proof}

\bigskip

The second part of the appendix deals with the fullness condition of a spatial process on positive  functions. As already discussed, it boils down to the following proposition.

\begin{prop}
Let $(h_n)_{n\in\IN}$  be a sequence in $\LtwoG$ that converges to $h\in\LtwoG$ and let $\fktC\in\Cu$. Then the composition $h_n\circ \fktC$ converges to $h\circ\fktC$.
\end{prop}

\begin{proof}
We follow the proof of Lemma 2.4 in \cite{LeSt} with some slightly modifications.

 As $\fktC\in\Cu$ there exists to every $\epsilon >0$ a $\delta >0$ such that $\abs{\fktC(x)-\fktC(y)}<\frac{\epsilon}{2}$ for all $x,y\in\IC$ with $\abs{x-y}<\delta$. We set 
$$A_{\delta}^n:=\{x\in X\ : \abs{h_n(x)-h(x)}\geq\delta\}$$
and split up the integral in the following way
\begin{align*}
\int_X \abs{\fktC\circ h_n - \fktC\circ h}^2 \,dm
= \int_{X\setminus A_{\delta}^n} \abs{\fktC\circ h_n - \fktC\circ h}^2 \,dm
+ \int_{A_{\delta}^n} \abs{\fktC\circ h_n - \fktC\circ h}^2 \,dm
\end{align*}
For the first part we obtain
\begin{align*}
\int_{X\setminus A_{\delta}^n} \abs{\fktC\circ h_n - \fktC\circ h}^2 \,dm
< \frac{\epsilon^2}{4} m(X\setminus A_{\delta}^n)
< \frac{\epsilon^2}{4}
\end{align*}
as $\abs{ h_n - h} < \delta$. For the second term we estimate
$$\abs{\fktC\circ h_n - \fktC\circ h} \leq 2\absinf{\fktC}$$
and obtain with the Tschebyscheff inequality 
$$\int_{A_{\delta}^n} \abs{\fktC\circ h_n - \fktC\circ h}^2 \,dm
\leq 4\absinf{\fktC}^2 m({A_{\delta}^n})
\leq \frac{4\absinf{\fktC}^2}{\delta^2}\int_X \abs{\fktC\circ h_n - \fktC\circ h}^2 \,dm.$$
Since $h_n$ converges to $h$ in $\LtwoG$, the last integral becomes arbitrary small for $n\to\infty$ such that 
$$\int_X \abs{\fktC\circ h_n - \fktC\circ h}^2 \,dm <\frac{\delta^2}{4\absinf{\fktC}^2} \frac{\epsilon}{2}$$
holds for $n$ sufficiently large.  Finally we have
$$\int_{X} \abs{\fktC\circ h_n - \fktC\circ h}^2 \,dm
< \frac{\epsilon^2}{4} + \frac{\epsilon}{2}<\epsilon,$$
which gives us the convergence of $\fktC\circ h_n$ to $\fktC\circ h$ in $\LtwoG$.
\end{proof}
\renewcommand{\refname}{Bibliographie}

\end{document}